\documentclass[a4paper,12pt]{article}

\usepackage[utf8]{inputenc}
\usepackage[T2A]{fontenc}
\usepackage[english,russian]{babel}

\usepackage{amsmath,amssymb,amsthm}
\usepackage[a4paper,hmargin=2.5cm,vmargin=2.5cm]{geometry}

\usepackage{array}
\usepackage{amssymb}
\usepackage{amsmath}
\usepackage{amsthm}
\usepackage{latexsym}
\usepackage{indentfirst}
\usepackage{bm}
\usepackage{enumerate}
\usepackage[dvips]{graphicx}
\usepackage{epsf}
\usepackage{wrapfig}
\usepackage{euscript}
\usepackage[dvipsnames]{xcolor}
\usepackage{dsfont}
\usepackage{cite}

\usepackage[colorlinks=true,linkcolor=red!50!black,pagecolor=blue!70!black,citecolor=green!40!black]{hyperref}

\newcommand{\pDiamond}{\ensuremath{\Diamond\!\!\!\!\hspace{1pt}\Diamond}}

\newcommand{\pw}[1]{\bar{#1}}

\newcommand{\con}{\wedge}

\newcommand{\imp}{\rightarrow}

\newcommand{\sameas}{\ensuremath{\leftrightharpoons}}

\newcommand{\vp}{\ensuremath{\varphi}}

\newtheorem{theorem}{\indent\bf Теорема}[section]
\newtheorem{corollary}[theorem]{\indent\bf Следствие}
\newtheorem{lemma}[theorem]{\indent\bf Лемма}
\newtheorem{sublemma}[theorem]{\indent\bf Подлемма}

\begin{document}

\title{Сложность фрагментов произведений с логикой {\bf Т} в языке с одной переменной\thanks{Работа К.\,И.~Александрова и М.\,Н.~Рыбакова поддержана программой <<Научный фонд НИУ ВШЭ>>, грант \mbox{21-04-027}.}}

\author{К.\,И.~Александров$^1$, М.\,Н.~Рыбаков$^2$, Д.\,П.~Шкатов$^3$ \medskip\\
        {\small $^{1,2}$НИУ ВШЭ} \\ {\small $^3$University of the Witwatersrand, Johannesburg}}

\date{}

\sloppy

\maketitle

\selectlanguage{english}

\begin{abstract}
We show that products of propositional modal logics containing the logic of
reflexive frames $\mathbf{T}$ as a factor are embeddable into their
single-variable fragments.
The proof is a simplified version of the proof, to appear, of a
similar result for
products and expanding relativized products containing as a factor the logic $\mathbf{KTB}$
of reflexive and symmetric Kripke frames.
%
\end{abstract}

\selectlanguage{russian}

\begin{abstract}
В работе показано, что любое произведение нормальных модальных пропозициональных логик, содержащее логику~$\mathbf{T}$ в качестве сомножителя, погружается в свой фрагмент от одной переменной. Приведённое доказательство является более простой версией аналогичного доказательства, которое готовится к публикации, для произведений и полупроизведений, содержащих в качестве сомножителя логику~$\mathbf{KTB}$ рефлексивно-симметричных шкал Крипке.
\end{abstract}

\section{Введение}

Известно, что <<естественные>> модальные и суперинтуиционистские логики, как правило, разрешимы~\cite{ChZ}, но часто имеют довольно высокую сложность проблемы разрешения~\cite{Ladner77,Statman79}, которая во многих случаях не понижается при ограничениях на средства языка~\cite{Spaan-1993-1, Halpern95, ChRyb03, Rybakov06, Rybakov07a, RShIGPL19} (хотя и не всегда~\cite{Hem01}); то же относится и к фрагментам полимодальных логик~\cite{Rybakov07, Rybakov08, RShIGPL18, RShICTAC18}. В этой работе мы затронем аналогичный вопрос в отношении произведений модальных логик~\cite{GKWZ,Kurucz08}, который до недавнего времени практически не был исследован.

Произведения пропозициональных модальных логик были введены в 1970\nobreakdash-х годах
Сегербергом~\cite{Seg73} и независимо Шехтманом~\cite{Shehtman78}. Они оказались тесно связаны с предикатными модальными логиками~\cite{GSH98, GKWZ}, а также нашли приложения в различных областях теоретической информатики~\cite{BCWZ02, KKZW07, FHMV95, HV89, WvdH03, GorSh09b, WZ99, WZ2000, LWZ08}.

Известно, что вычислительная сложность произведений модальных логик довольно высока~\cite{Marx99,GJL10,RZ01,GKWZ05,HHK02}, поэтому вопрос о том, как она изменится (и изменится ли), если мы будем вводить ограничения на используемые средства языка, является вполне закономерным.

Одно из возможных дополнительных требований к средствам языка состоит в том, чтобы ограничить сверху модальную глубину рассматриваемых формул~\cite{GJL10, MM01}. Другое~--- и именно оно нам будет здесь интересно~--- это ограничение числа переменных в формулах. Несмотря на то, что, как было отмечено выше, этот вопрос довольно хорошо исследован для мономодальных и полимодальных логик, результаты этих исследований не переносятся напрямую на произведения логик: мы, конечно же, можем гарантировать сложность фрагментов произведений не ниже сложности соответствующих фрагментов сомножителей, но такая оценка снизу является очень грубой.

Недавно было показано~\cite{RShJLC21a}, что произведения логик, содержащие в качестве сомножителя логику~$\mathbf{K}$, полиномиально погружаются в свой фрагмент от одной переменной. Доказательство~\cite{RShJLC21a} существенно использует тот факт, что в~$\mathbf{K}$ существует бесконечно много попарно неэквивалентных константных формул, и по этой причине не переносится на логики типа $\mathbf{T}$, $\mathbf{S4}$, $\mathbf{S5}$ и~др., где ситуация с константными формулами иная. Тем не менее, в готовящейся к публикации работе~\cite{KTB} был предложен метод получения похожих результатов и для некоторых расширений логики~$\mathbf{K}$, допускающих лишь конечные множества попарно неэквивалентных константных формул; этот метод подробно изложен для логики~$\mathbf{KTB}$, но он применим и ко многим сабфреймовым логикам, содержащимся в~$\mathbf{KTB}$, а возможно, и к некоторым другим. 

Здесь мы предлагаем прямое доказательство аналогичного результата для произведений с сомножителем~$\mathbf{T}$, используя чуть более простое моделирование, чем то, что предложено в~\cite{KTB}. Во-первых, мы уберём некоторые технические детали, получив более простой метод моделирования переменных в формуле формулами от одной переменной. При этом мы потеряем некоторые <<хорошие>> свойства моделирования, предложенного в~\cite{KTB}, но получим более прозрачную конструкцию. Во-вторых, возможность предлагаемого здесь моделирования может представлять интерес, например, в тех случаях, когда нам по каким-либо причинам важна антисимметричность отношения достижимости в шкалах Крипке (например, в $\mathbf{Grz}$-шкалах или $\mathbf{GL}$-шкалах; моделирование в~\cite{KTB} предполагает отсутствие требования антисимметричности); исходная идея такого моделирования предложена в~\cite{Halpern95}, и по сути мы лишь адаптируем её для произведений, содержащих логику $\mathbf{T}$ в качестве сомножителя. Именно, мы покажем, что любое произведение модальных логик, содержащих сомножитель~$\mathbf{T}$, экспоненциально погружается в свой фрагмент от одной переменной; экспоненциальная оценка является <<платой>> за упрощение конструкции из~\cite{KTB}, но мы сделаем замечание о том, как преобразовать предлагаемую здесь конструкцию в аналогичную, но с полиномиальной оценкой для функции моделирования всех переменных формулы формулами от одной переменной.


\section{Необходимые определения}

    Мы будем рассматривать $n$-модальный пропозициональный язык (для произвольного целого положительного~$n$), содержащий счётное множество пропозициональных переменных, константу
    $\bot$, булевы связки $\wedge$, $\vee$ и $\imp$, а также $n$ унарных модальных операторов $\Box_1, \ldots, \Box_n$. Формулы определяются обычным образом:
    $$
    \begin{array}{lcl}
    \varphi & ::= & \bot\mid p\mid(\vp\wedge\vp)\mid(\vp\vee\vp)\mid(\vp\to\vp)\mid\Box_k\vp,
    \end{array}
    $$
    где $p$ пробегает множество всех пропозициональных переменных языка, а $k$~--- множество $\{1,\ldots,n\}$.
    При записи формул мы будем использовать следующие стандартные сокращения:
    $\neg \vp = (\vp \imp \bot)$ и
    $\Diamond_i \vp = \neg \Box_i \neg \vp$ для любого
    $i \in \{1, \ldots, n\}$. 
    Когда скобки опущены, предполагается, что унарные операторы связывают формулы сильнее, чем $\wedge$ и $\vee$, которые, в свою очередь, связываю формулы сильнее, чем~$\imp$.  

    Модальную глубину $\mathop{\mathsf{md}}\varphi$ формулы $\vp$ определим обычно, т.е. как наибольшее число вложенных модальных операторов в~$\vp$.


    Как обычно, посредством $\mathbf{K}$ обозначим минимальную нормальную мономодальную (т.е. $1$-модальная) логику~\cite{ChZ}.

    Будем называть {\em $n$-шкалой Крипке\/} набор $\frak{F} = \langle W, R_1, \ldots, R_n \rangle$, где $W$~--- непустое множество {\em точек\/} и $R_1, \ldots, R_n$~--- бинарные {\em отношения достижимости\/} на $W$. Говорим, что точка $y$ {\em $R_i$\nobreakdash-достижима} из точки $x$, если $x R_i y$.


    Как обычно, шкала $\frak{F}' = \langle W', R'_1, \ldots, R'_n \rangle$ называется {\em подшкалой\/} шкалы $\frak{F} = \langle W, R_1, \ldots, R_n \rangle$ (обозначается $\frak{F}' \subseteq \frak{F}$), если $W' \subseteq W$ и $R'_i = R_i \upharpoonright W'$ для любого $i \in \{1, \ldots, n \}$.

    {\em Оценкой\/} на шкале $\frak{F} = \langle W, R_1, \ldots, R_n \rangle$ называется функция $v$, сопоставляющая каждой пропозициональной переменной подмножество $v(p)$ множества~$W$.

    Истинность формулы $\vp$ в точке $x$ шкалы $\frak{F}$ при оценке $v$ определяется рекурсивно:


    \medskip

    \hspace{-5pt}
    $ 
    \begin{array}{llcll}
        \bullet & \frak{F}, x \models^{v} p
            & \sameas
            & \mbox{$x \in v(p)$; }
            \medskip\\
        \bullet & \frak{F}, x \not\models^{v} \bot;
            \medskip\\
        \bullet & \frak{F}, x \models^{v} \vp_1 \imp \vp_2
            & \sameas
            & \mbox{$\frak{F},x\not\models^{v}\vp_1$ или $\frak{F},x\models^{v}\vp_2$;}
            \medskip\\
        \bullet & \frak{F}, x \models^{v} \Box_i \vp_1
            & \sameas
            & \mbox{$\frak{F}, y \models^{v} \vp_1$ для любого $y$: $x R_i y$.}
    \end{array}
    $

    \medskip


    Как обычно, формула $\vp$ называется истинной в $\frak{F}$ (пишем $\frak{F} \models \vp$), если $\frak{F}, x \models^v \vp$ для каждой точки $x$ из $\frak{F}$ и любой оценки~$v$. Множество формул $\varPhi$ считаем истинным в $\frak{F}$ (пишем $\frak{F} \models \varPhi$), если $\frak{F} \models \vp$ для любой формулы $\vp \in \varPhi$.

    Если $L$~--- модальная логика, то шкалу $\frak{F}$ называем {\it $L$-шкалой\/}, если $\frak{F} \models L$.

    Для класса $\frak{C}$ $n$-шкал Крипке определим $\mathbf{L}(\frak{C})$ как множество $n$-модальных формул, истинных в каждой шкале из $\frak{C}$; множество $\mathbf{L}(\frak{C})$ является нормальной $n$-модальной логикой и называется {\em логикой класса\/ $\frak{C}$};  $n$\nobreakdash-модальную логику $L$ называем {\em полной по Крипке\/}, если $L = \mathbf{L}(\frak{C})$ для некоторого непустого класса $\frak{C}$ $n$\nobreakdash-шкал.

    Мономодальную логику $\mathbf{T}$ определим как логику класса всех рефлексивных $1$-шкал~\cite{ChZ}.



    {\em Произведением\/} $1$\nobreakdash-шкал $\frak{F}_1 = \langle W_1, R_1 \rangle, \ldots, \frak{F}_n = \langle W_n, R_n \rangle$ называется $n$-шкала
    $$
    \begin{array}{lcl}
        \frak{F}_1 \times \ldots \times \frak{F}_n
        & = &
        \langle W_1 \times \ldots \times W_n, \bar{R}_1, \ldots, \bar{R}_n \rangle,
    \end{array}
    $$
    где для любого $i \in \{1, \ldots, n\}$,
    $$
    \begin{array}{lcl}
        \langle x_1, \ldots, x_n \rangle \bar{R}_i \langle y_1, \ldots, y_n
        \rangle
        & \leftrightharpoons &
                             \mbox{$x_i R_i y_i$ и $x_k = y_k$, для
                             любого $k \in \{1, \ldots, n\} \setminus \{i\}$.}
    \end{array}
$$

    Будем использовать следующую договоренность для обозначений: если $\pw{x} \in W_1 \times \ldots \times W_n$ и $s \in \{1, \ldots, n\}$, то $x_s$ является $s$-й координатой $\pw{x}$.

%
%

    {\em Произведением\/} полных по Крипке мономодальных логик $L_1, \ldots, L_n$ называется $n$-модальная логика
    $$
    \begin{array}{lcl}
        L_1 \times \ldots \times L_n & = & L (\{ \frak{F}_1 \times \ldots \times
        \frak{F}_n :\, \frak{F}_1 \models L_1, \ldots, \frak{F}_n \models L_n
        \}).
    \end{array}
    $$
    
    Сразу заметим, что произведение логик $L_1, \ldots, L_n$, перемноженных в другом порядке, даёт ту же логику с точностью до переобозначения модальностей. Поскольку мы планируем рассматривать произведения, содержащие сомножитель~$\mathbf{T}$, можем (и будем) считать, что $L_1 = \mathbf{T}$.

    Зафиксируем положительное целое $n$ и полные по Крипке логики $L_2,\ldots L_n$. Пусть во всех дальнейших рассуждениях $L_1 = \mathbf{T}$ и $L = L_1 \times L_2 \times \ldots \times L_n$.





\section{Фрагменты с одной переменной}

    Пусть $\vp$~--- $n$-модальная формула, содержащая только переменные $p_1, \ldots, p_m$, и пусть $p$~--- пропозициональная переменная, отличная от $p_1, \ldots, p_m$.

    Положим $\pDiamond_1 \psi = \Diamond_1 (\neg p \con \Diamond_1 (p \con \psi))$. Пусть для любого $k \in \{1, \ldots, m, m+1\}$
    $$
    \begin{array}{lcl}
         \alpha_k & = & \pDiamond_1^{k} \Box_1^{\phantom{i}} p;\\
         \beta_k & = & \neg p \con \Diamond_1^{\phantom{i}} (p\wedge \alpha_k). \\
    \end{array}
    $$
    Пусть также
    $$
    \begin{array}{lclcl}
         B  & = & \beta_{m+1} & = & \neg p \con \Diamond_1 (p\wedge \alpha_{m+1}).
    \end{array}
    $$

    Далее рекурсивно определим перевод~$\sigma$:
    $$
    \begin{array}{lcll}
        \sigma(p_k) & = & \beta_k, & \mbox{где $k \in \{1, \ldots, m \}$;} \\
        \sigma(\bot) & = & \bot;  \\
        \sigma(\psi \wedge \chi) & = & \sigma(\psi) \wedge \sigma(\chi);  \\
        \sigma(\psi \vee \chi) & = & \sigma(\psi) \vee \sigma(\chi);  \\
        \sigma(\psi \imp \chi) & = & \sigma(\psi) \imp \sigma(\chi);  \\
         \sigma(\Box_i  \psi) & = & \Box_i (B \imp \sigma(\psi)), & \mbox{где $i \in \{1, \ldots, n \}$}.
    \end{array}
    $$

Для каждого натурального $k$ определим модальности $\Box^{\leqslant k}$ и $\Box^{\leqslant k}_{-1}$, положив
    $$
    \begin{array}{lcl}
        \Box^{\leqslant 0} \psi & = & \psi;
        \\
        \Box^{\leqslant k+1} \psi & = & \Box^{\leqslant k}\psi\wedge\Box_1\Box^{\leqslant k}\psi\wedge \ldots \wedge\Box_n\Box^{\leqslant k} \psi;
        \\
        \Diamond^{\leqslant k} \psi & = & \neg\Box^{\leqslant k}\neg \psi;
        \\
        \Box^{\leqslant 0}_{-1} \psi & = & \psi;
        \\
        \Box^{\leqslant k+1}_{-1} \psi & = & \Box^{\leqslant k}_{-1}\psi\wedge\Box_2^{\phantom{i}}\Box^{\leqslant k}_{-1}\psi\wedge \ldots \wedge\Box_n^{\phantom{i}}\Box^{\leqslant k}_{-1} \psi.
    \end{array}
    $$

Определим формулу $A$, положив
    $$
    \begin{array}{lcl}
     A
    & =
    & \displaystyle
      B \con
     \Box^{\leqslant\mathop{\mathsf{md}}\vp}_{\phantom{i}} (B \imp \Box^{\leqslant\mathop{\mathsf{md}}\vp}_{-1} B)
     \con
     \Box^{\leqslant\mathop{\mathsf{md}}\vp}_{\phantom{i}} (\Diamond^{\leqslant \mathop{\mathsf{md}}\vp}_{-1} B \imp B).
    \end{array}
    $$
    Поскольку модальная глубина формулы $\vp$ ограничена сверху длиной формулы~$\vp$, формула $A$ может быть построена за время, ограниченное экспонентой от длины~$\vp$.

    \begin{lemma}
    \label{lem}
    $\varphi \in L \iff A \to \sigma (\varphi) \in L$.
    \end{lemma}

    \begin{proof}
    $(\Leftarrow)$ Предположим $\varphi \notin L$. Тогда  $\mathfrak{F}, \bar u \not\models^{v} \varphi$ для некоторого произведения $\mathfrak{F} = (\bar W, \bar R_1, \ldots, \bar R_n)$ шкал $\mathfrak{F}_1, \ldots, \mathfrak{F}_n$, таких, что $\mathfrak{F}_i = (W_i, R_i) \models L_i$ для любого $i \in \{1, 2, \ldots, n\}$, некоторой точки $\bar u \in \bar W$ и некоторой оценки $v$ на $\mathfrak{F}$.

    Построим произведение шкал $\mathfrak{F}'$ и оценку $v'$ на $\mathfrak{F}'$, опровергающую $A \rightarrow \sigma (\varphi)$.

    Для любого $k \in \{1, \ldots, m + 1\}$ определим $1$-шкалу $\mathfrak{B}_k = (U_k, S_k)$, положив $U_k = \{\langle v_i^k, w_i^k \rangle: 0 \leqslant i \leqslant k\}$ и взяв в качестве $S_k$ рефлексивное замыкание отношения $\{ \langle v_i^k, w_i^k \rangle: 0 \leqslant i \leqslant k\} \cup \{ \langle w_i^k, v_{i+1}^k \rangle: 0 \leqslant i \leqslant k - 1\}$.

    Для любого $x\in W_1$ определим шкалу $\mathfrak{B}_{k}^{x} = (U_k \times \{x\}, S_k^x)$, которая изоморфна $\mathfrak{B}_k$ при изоморфизме $f: w \mapsto \langle w, x \rangle$.

    Пусть $\mathfrak{F}_1' = (W_1',R_1')$, где
    $$
    \begin{array}{lcl}
    W_1' & = & \displaystyle W_1 \cup \bigcup\limits_{k=1}^{m+1} (U_k \times W_1), 
    \end{array}
    $$
    а $R_1'$~--- рефлексивное замыкание отношения
    $$
	R_1 \cup \bigcup\limits_{x \in W_1} \bigcup\limits_{k=1}^{m+1} S_k^x \cup \{\langle x, \langle v_0, x\rangle\rangle: x \in W_1\}.
	$$

    Тогда в $\mathfrak{F}_1'$ для любого $x \in W_1$, а также для любого $k \in \{1, \ldots, m+1\}$ корень $\langle v_0^k, x \rangle$ шкалы $\mathfrak{B}_k^x$ является $R_1'$-достижимым из точки~$x$.

    Положим $\mathfrak{F}' = \mathfrak{F}_1' \times \mathfrak{F}_2^{\phantom{i}} \times \ldots \times \mathfrak{F}_n^{\phantom{i}} = (\bar W', \bar R_1', \ldots, \bar R_n')$. Таким образом, $\mathfrak{F} \subseteq \mathfrak{F}'$. Также $R_1'$~--- рефлексивное отношение, поэтому $\mathfrak{F}_1' \models \mathbf{T}$ и $\mathfrak{F}' \models L$.

    Определим оценку $v'$ на шкале $\mathfrak{F}'$, полагая $\bar x \in v'(p)$, если выполнено хотя бы одно из следующих условий:
    \begin{itemize}
        \item $\exists i \in \{1, \ldots, m + 1\} ~\exists \bar z \in \bar W ~\bar x = (\langle w_i^{m+1}, z_1^{\phantom{i}} \rangle, z_2^{\phantom{i}}, \ldots, z_n^{\phantom{i}})$
        \item $\exists k \in \{1, \ldots, m\} ~\exists i \in \{1, \ldots, k\} ~\exists \bar z \in \bar W ~(\mathfrak{F},\bar z \models^v p_k^{\phantom{i}} ~\& ~\bar x = (\langle w_i^k, z_1^{\phantom{i}} \rangle, z_2^{\phantom{i}}, \ldots, z_n^{\phantom{i}}))$.
    \end{itemize}

   Заметим, что тогда формула $B$ истинна при оценке $v'$ в точности в тех точках шкалы $\mathfrak{F}'$, которые принадлежат множеству~$\bar W$. Отсюда сразу следует, что $\mathfrak{F}, \bar u \models^{v'} A$.
   
   Осталось показать, что $\mathfrak{F}, \bar u \not\models^{v'} \sigma (\varphi)$.

    Для того, чтобы это доказать, сначала заметим, что для любого $k \in \{1, \ldots, m\}$ формула $\beta_k$ ведет себя в $\mathfrak{F}'$ при оценке $v'$ так же, как переменная $p_k$ в $\mathfrak{F}$ при оценке~$v$. Именно, для любого $\bar y \in \bar W$ справедлива следующая эквивалентность:
        $$
        \begin{array}{lcl}
        \mathfrak{F}', \bar y \models^{v'} \beta_k 
          & \Longleftrightarrow 
          & \mathfrak{F}, \bar y \models^v p_k.
        \end{array}
        \eqno{\mbox{$(\ast)$}}
        $$
    Далее покажем, что для любой подформулы $\theta$ формулы $\vp$ и для любого $\bar x \in \bar W$ имеет место эквивалентность
    $$
    \begin{array}{lcl}
    \mathfrak{F},\bar x \models^v \theta 
      & \iff 
      & \mathfrak{F}', \bar x \models^{v'} \sigma (\theta),
    \end{array}
    $$
    которую можно обосновать индукцией по построению формулы~$\theta$.

    Если $\theta = \bot$, то утверждение очевидно; 
    если $\theta = p_k$, то $\sigma(\theta) = \beta_k$ и утверждение следует из~\mbox{$(\ast)$}. Этим обосновывается базис индукции.

    Обоснуем индукционный шаг.
    
    В тех случаях, когда $\theta = \psi \wedge \chi$, $\theta = \psi \vee \chi$ или $\theta = \psi \to \chi$, всё тривиально.

    Пусть $\theta = \Box_i\psi$. Тогда $\sigma(\theta) = \Box_i(B \rightarrow \sigma (\psi))$. Пусть $\bar x \in \bar W$.

    Предположим, что $\mathfrak{F}, \bar x \not\models^v \Box_i \psi$. Тогда $\mathfrak{F}, \bar y \not\models^v \psi$ для некоторого $\bar y \in \bar R_i(\bar x) \subseteq \bar W$. Заметим, что $\mathfrak{F}',\bar y \models^{v'} B$, т.к. $\bar y \in W$. По предположению индукции, $\mathfrak{F}', \bar y \not\models^{v'} \sigma(\psi)$. Нетрудно видеть, что $\bar y \in \bar R_i'(\bar x)$, т.к. $\mathfrak{F} \subseteq \mathfrak{F}'$. Тогда $\mathfrak{F}', \bar x \not\models^{v'} \Box_i(B \rightarrow \sigma(\psi))$.

    Пусть теперь $\mathfrak{F}', \bar x \not\models^{v'} \Box_i(B \rightarrow \sigma(\psi))$. Тогда $\mathfrak{F}', \bar y \models^{v'} B$ и $\mathfrak{F}', \bar y \not\models^{v'} \sigma(\psi)$ для некоторого $\bar y \in \bar R_i'(\bar x)$. Тогда $\bar y \in \bar W$, т.к. $\mathfrak{F}', \bar y \models^{v'} B$. По предположению индукции, $\mathfrak{F}', \bar y \not\models^v \psi$. Поскольку $\bar x \in \bar W$, получаем, что $\bar y \in \bar R_i(\bar x)$. Следовательно, $\mathfrak{F}, \bar x \not\models^{v} \Box_i \psi$.


    Из доказанного и того, что $\mathfrak{F}, \bar u \not\models^v \vp$ и $\bar u \in \bar W$, следует, что $\mathfrak{F'}, \bar u \not\models^{v'} \sigma(\vp)$.

    Как мы видели, $\mathfrak{F}' \models L$, а значит, $A \imp \sigma(\vp) \notin L$.

    $(\Rightarrow)$ Пусть $A \imp \sigma(\vp) \notin L$. Тогда $\mathfrak{F}, \bar u \not\models^v A \imp \sigma(\vp)$ для некоторого произведения $\mathfrak{F} = (\bar W, \bar R_1, \ldots, \bar R_n)$ шкал $\mathfrak{F}_1, \ldots, \mathfrak{F}_n$, такого, что $\mathfrak{F}_i = (W_i, R_i) \models L_i$ для любого $i \in \{1, 2, \ldots, n\}$, для некоторой точки $\bar u \in \bar W$ и некоторой оценки $v$ на $\mathfrak{F}$.

    Построим произведение шкал $\mathfrak{F}'$ и оценку $v'$ на $\mathfrak{F}'$, позволяющую опровергнуть формулу~$\vp$.

    Для каждого натурального $k$ определим отношения $\bar{R}^{\leqslant k}_{\phantom{i}}$ и $\bar{R}^{\leqslant k}_{-1}$, положив
    $$
    \begin{array}{lcl}
        \bar{R}^{\leqslant 0} & = & \{\langle \bar{x},\bar{x}\rangle : \bar{x}\in\bar{W}\};
        \\
        \bar{R}^{\leqslant k+1} & = & \bar{R}^{\leqslant k}\cup \bar{R}_{1} \circ \bar{R}^{\leqslant k} \cup \ldots \cup \bar{R}_{n} \circ \bar{R}^{\leqslant k};
        \\
        \bar{R}^{\leqslant 0}_{-1} & = & \{\langle \bar{x},\bar{x}\rangle : \bar{x}\in\bar{W}\};
        \\
        \bar{R}^{\leqslant k+1}_{-1} & = & \bar{R}^{\leqslant k}_{-1}\cup \bar{R}_{2}^{\phantom{i}} \circ \bar{R}^{\leqslant k}_{-1} \cup \ldots \cup \bar{R}_{n}^{\phantom{i}} \circ \bar{R}^{\leqslant k}_{-1};
        \\
    \end{array}
    $$

    Пусть
    $$
    \begin{array}{lcl}
        W'_1
        & =
        & \{ x_1^{\phantom{i}} \in W_1^{\phantom{i}} :
        \mbox{$\frak{F}, \pw{x} \models^v B$ и
        $\pw{x}\in \bar{R}^{\leqslant \mathop{\mathsf{md}}\vp}(u)$}\}.
    \end{array}
    $$

    Из того, что $\frak{F}, \pw{u} \models A$, следует, что $\frak{F}, \pw{u} \models B$. Кроме того, $\pw{u}\in \bar{R}^{\leqslant \mathop{\mathsf{md}}\vp}(u)$. Значит,
    $u_1^{\phantom{i}} \in W'_1$, и поэтому $W'_1 \ne \varnothing$.

    Пусть $R'_1 = R_1^{\phantom{i}} \upharpoonright W'_1$ и $\frak{F}'_1 = \langle W'_1, R'_1 \rangle$.

    Пусть
    $\frak{F}' = \frak{F}'_1 \times \frak{F}_2^{\phantom{i}} \times \ldots \times
    \frak{F}_n^{\phantom{i}} = \langle \bar{W}', \bar{R}'_1, \ldots, \bar{R}'_n
    \rangle$.

    Из того, что $\frak{F}'_1 \subseteq \frak{F}_1^{\phantom{i}}$, следует, что
    $\frak{F}' \subseteq \frak{F}$.  Напомним, что $L_1 = \mathbf{T}$, и это сабфреймовая логика, поэтому $\frak{F}'_1 \models L_1^{\phantom{i}}$; следовательно, $\frak{F}' \models L$.

    \begin{sublemma}
    \label{sublem:D}
        $\pw{y}\in \bar{W'}\cap \bar{R}^{\leqslant \mathop{\mathsf{md}}\vp}(u) ~\Longrightarrow~ \frak{F},\pw{y}\models^{v} B$.
    \end{sublemma}

    \begin{proof}
        Если $\pw{y}\in \bar{W'}\cap \bar{R}^{\leqslant \mathop{\mathsf{md}}\vp}(u)$, то существует такая точка $\pw{x}\in \bar{W'}$, что $x_1=y_1$, $\pw{x}\in \bar{R}^{\leqslant \mathop{\mathsf{md}}\vp}(u)$ и $\frak{F}, \pw{x} \models^v B$. Тогда ясно, что существует такая точка $z\in \bar{R}^{\leqslant \mathop{\mathsf{md}}\vp}_{-1}(x)\cap \bar{R}^{\leqslant \mathop{\mathsf{md}}\vp}_{-1}(y)$, что $z_1=x_1=y_1$. Второе условие для нас несущественно. Из первого же, с учётом того, что $\frak{F}, \pw{x} \models^v (B\to \Box^{\leqslant \mathop{\mathsf{md}}\vp}_{-1})$, заключаем, что $\frak{F}, \pw{x} \models^v B$, откуда, с учётом того, что $\frak{F}, \pw{y} \models^v (\Diamond^{\leqslant \mathop{\mathsf{md}}\vp}_{-1}\to B)$, заключаем, что $\frak{F},\pw{y}\models^{v} B$.
    \end{proof}

    Зададим оценку $v'$ на $\frak{F}'$ так, чтобы для каждого $k \in
    \{1, \ldots, m\}$ выполнялось условие
    $$
    \begin{array}{lcl}
        v'(p_k) & = & \{\pw{x}\in\bar{W}' :
        \frak{F},\pw{x}\models^{v} \beta_k\}.
    \end{array}
    $$

    Теперь заметим, что для любой подформулы
    $\theta$ формулы $\varphi$ и для любой точки
    $\pw{x}\in\bar{W}'\cap \bar{R}^{\leqslant \mathop{\mathsf{md}}\vp - \mathop{\mathsf{md}}\theta}(u)$
    $$
    \begin{array}{lcl}
        \frak{F}',\pw{x}\models^{v'} \theta
        & \iff
        & \frak{F},\pw{x}\models^{v}\sigma(\theta),
    \end{array}
    $$
    что несложно доказать индукцией по построению формулы~$\theta$.

    Если $\theta = {\bot}$, то $\sigma(\theta) = {\bot}$, и утверждение очевидно; если $\theta = p_k$, то $\sigma(\theta) = \beta_k$, и тогда утверждение следует из определения~$v'$. Этим обеспечивается базис индукции.

    Обоснуем индукционный шаг.
    
    В тех случаях, когда $\theta = \psi \wedge \chi$, $\theta = \psi \vee \chi$ или $\theta = \psi \to \chi$, всё тривиально.

    Пусть $\theta = \Box_i\psi$. Тогда $\sigma(\theta) = \Box_i(B \rightarrow \sigma (\psi))$. Пусть $\pw{x}\in\bar{W}'\cap \bar{R}^{\leqslant \mathop{\mathsf{md}}\vp - \mathop{\mathsf{md}}\theta}(u)$.

    Пусть $\frak{F}',\pw{x}\not\models^{v'} \Box_i \psi$. Тогда
    $\frak{F}',\pw{y}\not\models^{v'}\psi$ для некоторого $\bar{y}\in \bar{R}'_i(\bar{x})$.

    Из того, что $\frak{F}' \subseteq \frak{F}$ следует, что
    $\bar{y}\in \bar{R}_i(\bar{x})$, а значит, $\pw{y}\in\bar{W}'\cap \bar{R}^{\leqslant \mathop{\mathsf{md}}\vp - \mathop{\mathsf{md}}\psi}(u)$.

    Тогда, по предположению индукции,
    $\frak{F},\pw{y}\not\models^{v} \sigma(\psi)$, а по подлемме~\ref{sublem:D}, $\frak{F},\pw{y}\models^{v} B$. Значит,
    $\frak{F},\pw{x}\not\models^{v} \Box_i (B \imp \sigma(\psi))$.

    Наоборот, предположим что
    $\frak{F},\pw{x}\not\models^{v} \Box_i (B \imp \sigma(\psi))$.  Тогда
    $\frak{F},\pw{y} \models^{v} B$ и
    $\frak{F},\pw{y}\not\models^{v}\sigma(\psi)$ для некоторого
    $\bar{y}\in \bar{R}_i(\bar{x})$.  Ясно, что $\pw{y}\in\bar{W}'\cap \bar{R}^{\leqslant \mathop{\mathsf{md}}\vp - \mathop{\mathsf{md}}\psi}(u)$. Таким образом, $y_1^{\phantom{i}} \in W'_1$ по определению $W'_1$; но тогда $\bar{y} \in \bar{W}'$.  Поэтому, по предположению
    индукции, $\frak{F}',\pw{y}\not\models^{v'} \psi$. Тогда из того, что $\bar{x}\in \bar{W}'$, следует, что $\bar{y}\in \bar{R}'_i(\bar{x})$.  Значит,
    $\frak{F}',\pw{x}\not\models^{v'} \Box_i \psi$.

    Как мы видели, $\bar{u} \in \bar{W}'$, и поэтому
    $\frak{F}', \bar{u} \not\models^{v'} \vp$, а поскольку $\frak{F}' \models L$, получаем, что $\vp \notin L$.
    \end{proof}

    Прямым следствием леммы~\ref{lem} является следующая теорема.

    \begin{theorem}
    \label{th:T}
        Любое произведение логик вида $\mathbf{T}\times L_2 \times \ldots \times L_n$ экспоненциально погружается в собственный фрагмент от одной переменной.
    \end{theorem}

\section{Замечания}

Теорема~\ref{th:T} сразу даёт следствия о неразрешимости фрагментов от одной переменной для всех неразрешимых произведений логик, в которых имеется сомножитель~$\mathbf{T}$. Заметим, что в рассуждениях выше логику $\mathbf{T}$ можно заменить на $\mathbf{K}$: для возникающих шкал уже не нужно будет проверять условие рефлексивности и брать рефлексивное замыкание возникающих в построениях отношений достижимости, а свойство сабфреймовости для $\mathbf{K}$ тоже выполняется. В итоге получаем, что аналог теоремы~\ref{th:T} справедлив и для~$\mathbf{K}$. Конечно, такой результат для произведений с~$\mathbf{K}$ является более слабым, чем представленный в~\cite{RShJLC21a}, но он, тем не менее, имеет, как минимум, одно преимущество. В~\cite{HHK02} доказано, что любая логика между $\mathbf{K}\times\mathbf{K}\times\mathbf{K}$ и $\mathbf{S5}\times\mathbf{S5}\times\mathbf{S5}$ неразрешима, при этом для доказательства используется одно и то же сведение некоторой неразрешимой проблемы к каждой такой логике. Применяя к этому сведению описанную выше технику, получаем следующий результат.

    \begin{corollary}
    \label{cor:th:T}
        Фрагмент от одной переменной любой логики, заключённой между\/ $\mathbf{K}\times\mathbf{K}\times\mathbf{K}$ и\/ $\mathbf{T}\times\mathbf{S5}\times\mathbf{S5}$, неразрешим.
    \end{corollary}
    
Теперь обратимся к вопросу о том, можно ли как-то изменить $\sigma$, чтобы получить полиномиальное погружение логик, содержащих сомножитель $\mathbf{T}$, в свой фрагмент от одной переменной. И ответ здесь такой: да, можно. Суть изменения в том, что вместо модальностей $\Box^{\leqslant\mathop{\mathsf{md}}\vp}_{\phantom{i}}$ и $\Box^{\leqslant\mathop{\mathsf{md}}\vp}_{-1}$, дающих экспоненциальную зависимость длины формулы $A$ от длины формулы~$\varphi$, нужно использовать другие модальности, сходные с указанными, но учитывающие то, какие именно последовательности вложенных модальностей имеются в формуле~$\vp$. Это приводит к более сложным построениям, увеличивая объём доказательства (при сходной детализации) в несколько раз. Именно такой подход и изложен в~\cite{KTB}; мы же хотели показать здесь более простой вариант построений, убрав технические детали, работающие на полиномиальность возникающего погружения, и оставив лишь ключевую идею.


\end{document}